\documentclass[11pt]{amsart}

\usepackage{amsmath, amsfonts, amssymb,amsthm}
\usepackage{amstext}
\usepackage{mathrsfs}

\usepackage{float,epsf,subfigure}
\usepackage[all,cmtip]{xy}
\usepackage{hyperref}
\usepackage{algorithm}
\usepackage{algorithmic}
\usepackage{mathtools}
\usepackage{float}
\usepackage{bm}
\theoremstyle{plain}
\newtheorem{theorem}{Theorem}[section]

\newtheorem{cor}[theorem]{Corollary}

\newtheorem{def-thm}[theorem]{Definition-Theorem}
\newtheorem{lemma}[theorem]{Lemma}

\newtheorem*{tha}{Theorem A}

\theoremstyle{definition}

\newtheorem{remark}[theorem]{Remark}

\begin{document}
\title[A refined   form of  the  second main theorem]{A refined   form of  the  second main theorem  on complete non-positively curved K\"ahler manifolds}
\author[X. Dong]
{Xianjing Dong}

\address{School of Mathematical Sciences \\ Qufu Normal University \\ Qufu, Jining, Shandong, 273165, P. R. China}
\email{xjdong05@126.com}

  %\thanks{This work was partially supported by the
%NSFC (No. 11301260,11101201), the NSF of Jiangxi (No.
%20132BAB211003) and the YFED of Jiangxi (No. GJJ13078) of China.}

\subjclass[2010]{32H30} 
\keywords{Nevanlinna theory;  second main theorem; K\"ahler manifolds; non-positively curved; Green function.}
%\thanks{The author was supported in part by the Simon Foundation (Grant No. 531604)}
\date{}
\maketitle \thispagestyle{empty} \setcounter{page}{1}

\begin{abstract}  How to devise   a  second main theorem with best  error terms is a central   problem   in the study of  Nevanlinna theory. 
However, it seems   difficult  to be done 
for a general  non-positively curved K\"ahler manifold. 
Based on the  work of  A. Atsuji in Nevanlinna theory, 
we present   a refined form of 
 the second main theorem  of meromorphic mappings 
 on  a general complete  K\"ahler manifold 
   with non-positive sectional curvature using  a good   estimate. 
      This result  improves   
       the   error terms in the second main theorem  obtained   by A. Atsuji in 2018. 
  \end{abstract}

%\tableofcontents
\medskip

%\newpage
\setlength\arraycolsep{2pt}
\medskip

\section{Introduction}

The study of Nevanlinna theory on complete  K\"ahler manifolds with non-positive sectional curvature is  an important   and  difficult task. 
As early as  the 1990s, A. Atsuji \cite{at1995} began  his outstanding research  work  in this direction   by developing  the probabilistic technique initialized by T. K.  Carne \cite{carne}. 
Later, he wrote a series of papers (see \cite{at2005, at2008a, at2008b, at2010, at2018a, at2018b}) 
concerning  the second main theorem of meromorphic functions on such manifolds.  In Nevanlinna theory, 
how to establish     a second main theorem with optimal  
 error terms  is a central   problem. 
    During   a long time,  Atsuji   made   efforts  to    improve  the  error terms for a 
       complete non-positively curved K\"ahler manifold  
       (see  \cite{at2005, at2008b, at2018a}). 
           In 2018,  
he gave  a  main error term $O(-\kappa(r)r^2)$   in \cite{at2018a}, 
in which  $\kappa(r)$ is the lower bound of   Ricci curvature on a geodesic ball with radius $r$ 
(see  (\ref{ricci}) below).   
 Following  Atsuji, 
  the  author \cite{Dong}  also extended      
  Carlson-Griffiths' second main theorem  \cite{gri}
  to  a  complete non-positively curved 
     K\"ahler manifold with 
           the  same  error term $O(-\kappa(r)r^2).$
       To make  it  more  clear,  we  would  give  a  brief introduction to  Atsuji-Dong's   second main theorem (see \cite{at2018a, Dong}) as follows. 

Let $M$ be an $m$-dimensional  complete non-compact K\"ahler manifold with non-positive sectional curvature. 
Let $\mathscr R$ stand for   the Ricci form of $M.$   Fix a reference point $o\in M.$  Denote by  $B(r)$   the geodesic ball centered at $o$ with radius $r$ in $M.$   Set 
\begin{equation}\label{ricci}
 \kappa(r)=\frac{1}{2m-1}\inf_{x\in B(r)}R(x),
 \end{equation}
   where 
    $$R(x)=\inf_{X\in T_xM, \ \|X\|=1}{\rm{Ric}}(X, X)$$
    is   the pointwise lower bound of Ricci curvature of $M.$
 Let $X$ be a complex projective manifold with $\dim X\leq \dim M,$ over which one   puts 
     a Hermitian positive  line bundle $(L, h).$ Let $f: M\to X$ be a meromorphic mapping. Let $D\in|L|,$ where $|L|$ is the complete linear system of $L.$ 
       Referring to  Section 2,   we have the Nevanlinna's functions $T_f(r, L), m_f(r, D), \overline{N}_f(r,D), T_f(r, K_X)$ and $T(r, \mathscr R)$  on $B(r).$

  \begin{tha}[Atsuji-Dong]\label{th1}  Let $M$ be a  
     complete non-compact  K\"ahler manifold   with non-positive sectional curvature.  
Let $X$ be a smooth complex projective variety  of complex dimension not greater than  that  of $M.$
 Let $D\in|L|$ be a reduced divisor of simple normal crossing type, where $L$ is a positive line bundle over $X.$   Let  $f:M\rightarrow X$ be a  differentiably non-degenerate meromorphic mapping.  Then  for any  $\delta>0,$ there exists a subset $E_\delta\subset(1, \infty)$ of finite Lebesgue measure such that 
        \begin{eqnarray*}
&&T_f(r,L)+T_f(r, K_X)+T(r, \mathscr R) \\
&\leq& \overline N_f(r,D)+O\left(\log^+ T_{f}(r,L)-\kappa(r)r^2+\delta\log r\right)
         \end{eqnarray*}
    holds for all  $r>1$ outside $E_\delta,$ where $\kappa(r)$ is defined by $(\ref{ricci}).$     
\end{tha}  

Based on  an upper estimate of the first existing  time  for $B(r)$ of   Brownian motion  started  at $o$ in $M$,  the curvature term  $T(r, \mathscr R)$ is bounded  from below by $O(\kappa(r)r^2)$ (see \cite{Dong}, Lemma 5.6).
It is clear   that  Theorem A extends   the classical second main theorem for $\mathbb C^m.$
However, this theorem   fails to recover the classical second main theorem for the unit complex  ball $\mathbb B$ in  $\mathbb C^m,$ 
because    its main error term  is $O(r^2),$ while the optimal   main error term is $O(r)$ under the Poincar\'e metric of  sectional curvature -1. 
Therefore, the main error term $O(-\kappa(r)r^2)$ in Theorem A  is  rough. 

In this paper, our  aim   is to  refine    this  second main theorem, i.e.,  Theorem A.  By setting up  a new calculus lemma  (see  Theorem \ref{cal}),  
 we receive  a good main error term $O(\sqrt{-\kappa(r)}r)$ in the second main theorem. 
In what follows, we  introduce the main results in the present paper. 

  \begin{theorem}\label{main1}  Let $M$ be a  
     complete non-compact  K\"ahler manifold   with non-positive sectional curvature.  
Let $X$ be a smooth complex projective variety  of complex dimension not greater than  that  of $M.$
 Let $D\in|L|$ be a reduced divisor of simple normal crossing type, where $L$ is a positive line bundle over $X.$   Let  $f:M\rightarrow X$ be a  differentiably non-degenerate meromorphic mapping.  Then  for any  $\delta>0,$ there exists a subset $E_\delta\subset(1, \infty)$ of finite Lebesgue measure such that 
        \begin{eqnarray*}
&&T_f(r,L)+T_f(r, K_X)+T(r, \mathscr R) \\
&\leq& \overline N_f(r,D)+O\left(\log^+ T_{f}(r,L)+\sqrt{-\kappa(r)}r+\delta\log r\right)
         \end{eqnarray*}
    holds for all  $r>1$ outside $E_\delta,$  where $\kappa(r)$ is defined by $(\ref{ricci}).$          
\end{theorem}  

When   $M$ has   constant sectional curvature, we further show  that 
 
  \begin{cor}\label{main2}  Let $M$ be a  
     complete non-compact  K\"ahler manifold   with constant  sectional curvature $\kappa.$   
Let $X$ be a smooth complex projective variety  of complex dimension not greater than  that  of $M.$
 Let $D\in|L|$ be a reduced divisor of simple normal crossing type, where $L$ is a positive line bundle over $X.$   Let  $f:M\rightarrow X$ be a  differentiably non-degenerate meromorphic mapping.   Then  for any  $\delta>0,$ there exists a subset $E_\delta\subset(1, \infty)$ of finite Lebesgue measure such that 
        \begin{eqnarray*}
T_f(r,L)+T_f(r, K_X) 
&\leq& \overline N_f(r,D)+O\left(\log^+ T_{f}(r,L)+\sqrt{-\kappa}r+\delta\log r\right)
         \end{eqnarray*}
    holds for all  $r>1$ outside $E_\delta.$     
\end{cor}  

It is clear  that   Corollary \ref{main2} covers the classical second main theorem for $\mathbb B.$  
 Finally,  we consider a defect relation. 
  The simple defect  of $f$ with respect to $D$  is defined   by
$$ \bar\delta_f(D)=1-\limsup_{r\rightarrow\infty}\frac{\overline{N}_f(r,D)}{T_f(r,L)}.$$
Set 
$$\left[\frac{c_1(K_X^*)}{c_1(L)}\right]=\inf\left\{s\in\mathbb R: \ \eta\leq s\omega;  \  \   ^\exists\eta\in c_1(K^*_X), \ ^\exists\omega\in c_1(L) \right\}.$$
\begin{cor}  Assume the same conditions as in Theorem $\ref{main1}.$ If 
$$\limsup_{r\to\infty}\frac{\sqrt{-\kappa(r)}r-T(r,\mathscr R)}{T_f(r,L)}=0,$$ 
then 
$$\bar\delta_f(D)
\leq \left[\frac{c_1(K_X^*)}{c_1(L)}\right].
$$
\end{cor}

\section{Notations}

Let $M$ be a complete non-compact K\"ahler manifold 
 of complex dimension $m,$ with  
    K\"ahler form   $\alpha$  associated to the  K\"ahler metric $g=(g_{i\bar j})$   defined by  
 $$\alpha=\frac{\sqrt{-1}}{\pi}\sum_{i,j=1}^mg_{i\bar j}dz_i\wedge d\bar z_{j}$$
  in  holomorphic local  coordinates $z_1,\cdots,z_m.$    The Ricci form of $M$ is defined by 
$$\mathscr R=-dd^c\log\det(g_{i\bar j}),$$ 
where 
    $$dd^c=\frac{\sqrt{-1}}{2\pi}\partial\overline{\partial}.$$
\ \ \ \    Fix a reference point $o\in M.$   Let  
   $B(r)=\{x\in M: \rho(x)<r\}$ stand for  the geodesic ball centered at $o$ with radius $r$ in $M,$  where  $\rho(x)$  is  the Riemannian distance function on $M$  from $o.$ 
 Denote by     $\Delta$      the Laplace-Beltrami operator   on $M.$  Let    $g_r(o,x)$ stand for   the Green function of $\Delta/2$ for $B(r)$ with a pole at $o$ satisfying Dirichlet boundary condition. 
Then, the harmonic measure $\pi_r$ on $\partial B(r)$ (the boundary of $B(r)$) with respect to $o$ can be  expressed  as 
 $$d\pi_r=\frac{1}{2}\frac{\partial g_r(o,x)}{\partial{\vec{\nu}}}d\sigma_r,$$
  where  $\partial/\partial \vec\nu$ is the inward  normal derivative on $\partial B(r),$ $d\sigma_{r}$ is the Riemannian area element  of 
$\partial B(r).$

  Now, we   define  the Nevanlinna's functions. 
   Let $X$ be a complex projective manifold, over which one   puts 
     a Hermitian positive  line bundle $(L, h)$
    with   Chern form  $$c_1(L,h):=-dd^c\log h>0.$$ 
    
      Let $f: M\to X$ be a meromorphic mapping. The characteristic function of $f$ with respect to $L$ is defined by 
    \begin{eqnarray*}
 T_f(r, L) &=& \frac{\pi^m}{(m-1)!}\int_{B(r)}g_r(o,x)f^*c_1(L,h)\wedge\alpha^{m-1}  \\
 &=&    -\frac{1}{4}\int_{B(r)}g_r(o,x)\Delta\log(h\circ f)dv,
     \end{eqnarray*}
     where $dv$ is the Riemannian volume element of $M.$
     
   Let $s_D$ be the  canonical section  associated to $D\in|L|$  (i.e., $s_D$ is a holomorphic section of $L$ over $X$ with zero divisor $D$),  where $|L|$ is the complete linear system of $L.$  
The proximity function of $f$ with respect to $D$ is defined  by
 $$m_f(r,D)=\int_{\partial B(r)}\log\frac{1}{\|s_D\circ f\|}d\pi_r.$$
Define the  counting function  and   simple counting function of $f$ with respect to $D$ respectively  by 
   \begin{eqnarray*}
N_f(r,D)&=& \frac{\pi^m}{(m-1)!}\int_{f^*D\cap B(r)}g_r(o,x)\alpha^{m-1}, \\
\overline{N}_f(r, D)&=& \frac{\pi^m}{(m-1)!}\int_{f^{-1}(D)\cap B(r)}g_r(o,x)\alpha^{m-1}. 
 \end{eqnarray*}
Moreover,  the characteristic function of   $\mathscr R$  is defined by
     \begin{eqnarray*}
 T(r,\mathscr R)&=& \frac{\pi^m}{(m-1)!}\int_{B(r)}g_r(o,x)\mathscr R\wedge\alpha^{m-1} \\
 &=&  -\frac{1}{4}\int_{B(r)}g_r(o,x)\Delta\log\det(g_{i\bar j})dv.  \     
     \end{eqnarray*}
\ \ \ \    Jensen-Dynkin formula (see, e.g.,  \cite{at2008a, bass, Dong,  Ni}) plays a central  role, which is stated as follows. 
\begin{lemma}[Jensen-Dynkin formula]\label{dynkin} Let $\phi$ be a $\mathscr C^2$-class function on  $M$ outside a polar set of singularities at most. Assume that $\phi(o)\not=\infty.$  Then
$$\int_{\partial B(r)}\phi(x)d\pi_{r}(x)-\phi(o)=\frac{1}{2}\int_{B(r)}g_r(o,x)\Delta \phi(x)dv(x).$$
\end{lemma}

Combining Jensen-Dynkin formula with Poincar\'e-Lelong formula (see \cite{No, ru}),  we can obtain the following first main theorem. 

\begin{theorem}[\cite{at2018a, Dong}]  Assume   that $f(o)\not\in{\rm{Supp}}D.$ Then
$$T_f(r,L)+\log\frac{1}{\|s_D\circ f(o)\|}=m_f(r,D)+N_f(r,D).$$
\end{theorem}

\section{Lemmas for Non-positive Sectional Curvature}

From now on,  assume that  $M$ is a complete non-compact K\"ahler manifold 
 with non-positive sectional  curvature.   We shall use  the same notations  given as before. 
In the study of value distribution theory, 
 we may   assume  without loss of generality  that  $M$ is simply-connected 
   for a  technical reason, 
   since one  may   consider the universal  covering $\pi: \tilde M\to M.$
 As long as   $\tilde M$  is equipped  with the pull-back metric $\tilde g$  induced from $g$ by $\pi,$  
     then $(\tilde M, \tilde g)$ keeps the  same  curvature as  one of $(M, g).$ 
 Recall that  $\kappa(r)$ is the lower bound of  the Ricci curvature of $M$ on $B(r)$  defined by  (\ref{ricci}) in Introduction.   

    Put 
  \begin{equation}\label{chi}
 \chi(s, t)=\left\{
                \begin{array}{ll}
t, \  \ & s=0; \\
  \frac{\sinh s t}{s}, \ \ & s\not=0.
              \end{array}
              \right.
               \end{equation}
 Consider  the following  Jacobi equation on $[0,\infty)$:
   \begin{equation}\label{G}
 G''+\kappa(t)G=0; \ \ \ \    G(0)=0, \ \ \ \  G'(0)=1.
   \end{equation}
 Note that this equation  has a unique  continuous solution. 

We give two-sided estimates of $G(t)$ as follows. 
\begin{lemma}\label{est} For $t\geq0,$ we have 
   $$ \chi(0, t)\leq G(t)\leq  \chi\left(\sqrt{-\kappa(t)}, t\right),$$
   where $\chi(s,t)$ is defined by $(\ref{chi}).$
\end{lemma}

\begin{proof}   Treat   the following  Jacobi equation on $[0,\infty)$:
 $$H_1''=0; \ \ \ \    H_1(0)=0, \ \ \ \  H_1'(0)=1,$$
which is uniquely solved by  $H_1(t)=t.$ Since  $\kappa(t)\leq0$ for $t\geq0, $   the standard comparison argument in ODEs  shows that   
$$G(t)\geq H_1(t)=t=\chi(0,t).$$
Fix any  number $t_0>0.$ It is clear  that $\kappa(t)\geq\kappa(t_0)$ for $0\leq t\leq t_0.$
Treat   the following initial value  problem on $[0, t_0]$:
 $$H_2''+\kappa(t_0)H_2=0; \ \ \ \    H_2(0)=0, \ \ \ \  H_2'(0)=1,$$
which gives a unique solution $H_2(t)=\chi(\sqrt{-\kappa(t_0)}, t).$ Applying  the standard comparison argument (see \cite{P1}, Theorems 1, 2 for instance), 
we deduce that 
 $$G(t)\leq H_2(t)=\chi\left(\sqrt{-\kappa(t_0)},t\right)$$
for $0\leq t\leq t_0.$  In particular,  
  $G(t_0)\leq \chi(\sqrt{-\kappa(t_0)},t_0).$ 
Using the arbitrariness  of $t_0,$ we deduce that 
    $$G(t)\leq \chi\left(\sqrt{-\kappa(t)},t\right)$$ for $t\geq0.$ This completes the proof. 
\end{proof}

 By using stochastic calculus,  Atsuji obtained   a lower bound of  $g_r(o,x).$  
    
    \begin{lemma}[\cite{at2018a}]\label{zz} Let $\eta>0$ be any  number. Then,  there exists  a constant $c_1>0$ such that
  $$g_r(o,x)\geq c_1\frac{\displaystyle\int_{\rho(x)}^rG(t)^{1-2m}dt}{\displaystyle\int_{\eta}^rG(t)^{1-2m}dt}$$
  holds for all $x\in B(r)\setminus \overline{B(2\eta)}$ with $r>3\eta.$ In particular, if $M$ is non-parabolic, then 
  there exists a constant $c_2>0$ such that 
  $$g_r(o,x)\geq c_2\int_{\rho(x)}^rG(t)^{1-2m}dt$$
  holds for all $x\in B(r),$ where $G(t)$ is defined by $(\ref{G}).$
  \end{lemma}
    
    For an upper bound of $g_r(o,x),$ we have the following well-known fact. 
    
 \begin{lemma}[\cite{Deb}]\label{sing}  We have 
$$g_r(o,x)\leq\left\{
                \begin{array}{ll}
                  \frac{1}{\pi}\log\frac{r}{\rho(x)}, \  \   & m=1; \\
                  \frac{1}{(m-1)\omega_{2m-1}}\big{(}\rho(x)^{2-2m}-r^{2-2m}\big{)},  \ \     & m\geq2  \\
                \end{array}
              \right. $$
              and
$$d\pi_r(x)\leq\frac{r^{1-2m}}{\omega_{2m-1}}d\sigma_r(x),  \  \ \ \ \ \ \ \ \ \ \ \ \  \ \   \ \ \ \ \  \ \ 
\  \ \ \ \ \ \ \ \ \ \ \ \  $$
 where  $\omega_{2m-1}$ is the  standard Euclidean    area  of  the unit sphere in $\mathbb R^{2m}.$ 
\end{lemma}

We also need the  Borel's growth lemma (see \cite{No, ru}) as follows. 

 \begin{lemma}[Borel's Growth  Lemma]\label{} Let $u\geq0$ be a non-decreasing  function  on $(r_0, \infty)$ with $r_0\geq0.$ Then for any $\delta>0,$ there exists a subset $E_\delta\subset(r_0,\infty)$
 of finite Lebesgue measure such that  
 $$u'(r)\leq u(r)^{1+\delta}$$
 holds for all $r>r_0$ outside $E_{\delta}.$  
 \end{lemma}
 \begin{proof} The conclusion is clearly true  for $u\equiv0.$ Next, we assume that $u\not\equiv0.$
 Since $u\geq0$ is a non-decreasing  function,  there exists a number $r_1>r_0$ such that $u(r_1)>0.$  The non-decreasing property of $u$ implies that  the  limit
$$\eta:=\lim_{r\to\infty}u(r)$$
exists or $\eta=\infty.$  If $\eta=\infty,$ then $\eta^{-1}=0.$ 
 Set  
 $$E_\delta=\left\{r\in(r_0,\infty):  u'(r)>u(r)^{1+\delta}\right\}.$$
Since $u$ is a non-decreasing function on $(r_0, \infty),$  we deduce that  $u'(r)$ exists for  almost all  $r\in(r_0, \infty).$   It is therefore  
   \begin{eqnarray*}
\int_{E_\delta}dr 
 &\leq& \int_{r_0}^{r_1}dr+\int_{r_1}^\infty\frac{u'(r)}{u(r)^{1+\delta}}dr \\
 &=&\frac{1}{\delta u(r_1)^\delta}-\frac{1}{\delta \eta^\delta}+r_1-r_0\\
 &<&\infty.
    \end{eqnarray*}
 This completes the proof. 
 \end{proof}
 
  Set 
  \begin{equation}\label{kkk}
  K(r,\delta)=\frac{r^{1-2m}\left(\displaystyle\int_{\frac{1}{3}}^rG(t)^{1-2m}dt\right)^{(1+\delta)^2}}{G(r)^{(1-2m)(1+\delta)}},
    \end{equation}
where $G(t)$ is defined by (\ref{G}).

With the previous preparations,  we  establish a calculus lemma: 
\begin{theorem}[Calculus Lemma]\label{cal} Let $k\geq0$ be a locally integrable function on $M.$ Assume that $k$ is locally bounded at $o.$ Then 
 for any   $\delta>0,$  there exists   a subset $E_{\delta}\subset(1,\infty)$ of finite Lebesgue measure such that
$$\int_{\partial B(r)}kd\pi_r\leq C K(\delta, r)
      \left(\int_{B(r)}g_r(o,x)kdv\right)^{(1+\delta)^2}$$
holds for all  $r>1$ outside $E_{\delta},$ where    $K(r,\delta)$ is defined by $(\ref{kkk})$ and $C>0$ is a sufficiently large  constant independent of  $\delta, r.$ 
 \end{theorem}
\begin{proof} 
 For $r>1,$  it yields from Lemma \ref{sing} that 
      \begin{eqnarray*}
\Lambda(r) &:=& \int_{B(r)}g_r(o,x)kdv  \\
&\geq& \int_1^rdt\int_{\partial B(t)}g_r(o,x)kd\sigma_t \\
      &\geq& \omega_{2m-1} \int_1^rt^{2m-1}dt\int_{\partial B(t)}g_r(o,x)kd\pi_t.  
      \end{eqnarray*}
       According to Lemma \ref{zz},  there exists  a constant $c>0$ such that
  $$g_r(o,x)\geq c\frac{\displaystyle\int_{\rho(x)}^rG(t)^{1-2m}dt}{\displaystyle\int_{1/3}^rG(t)^{1-2m}dt}$$
  holds all $x\in B(r)\setminus \overline{B(2/3)}$ with $r>1.$
Whence, for $r>t>1$  
$$g_r(o,x)\big|_{\partial B(t)}\geq c\frac{\displaystyle\int_{t}^rG(s)^{1-2m}ds}{\displaystyle\int_{\frac{1}{3}}^rG(s)^{1-2m}ds}.$$  
       It is therefore 
                 \begin{eqnarray*}
   & &   \int_{\partial B(t)}g_r(o,x)kd\pi_t  \\
      &\geq& c\left(\displaystyle\int_{\frac{1}{3}}^rG(s)^{1-2m}ds\right)^{-1} \int_{\partial B(t)}kd\pi_t\displaystyle\int_{t}^rG(s)^{1-2m}ds. 
                       \end{eqnarray*}
Combining the above to get  
      \begin{eqnarray*}
&& \Gamma(r) \\ 
      &\geq& c\omega_{2m-1}\left(\displaystyle\int_{\frac{1}{3}}^rG(s)^{1-2m}ds\right)^{-1} \int_1^rt^{2m-1}dt\int_{\partial B(t)}kd\pi_t\displaystyle\int_{t}^rG(s)^{1-2m}ds \\
            &:=& \Lambda(r), 
      \end{eqnarray*}
    which  leads to  
                           \begin{eqnarray*}
\frac{\displaystyle\frac{d}{dr}\left(\Lambda(r)\displaystyle\int_{\frac{1}{3}}^rG(t)^{1-2m}dt\right)}{G(r)^{1-2m}} &=&  c\omega_{2m-1}\int_1^rt^{2m-1}dt\int_{\partial B(t)}kd\pi_t.
                       \end{eqnarray*}
In further, we obtain 
       $$ \frac{d}{dr}\frac{\displaystyle\frac{d}{dr}\left(\Lambda(r)\displaystyle\int_{\frac{1}{3}}^rG(t)^{1-2m}dt\right)}{G(r)^{1-2m}}=c\omega_{2m-1}r^{2m-1}\int_{\partial B(r)}kd\pi_r.$$                
       Using  Borel's growth lemma twice,  then for any $\delta>0,$  there exists a subset $E_\delta\subset(1,\infty)$ of finite Lebesgue measure such that   
            \begin{eqnarray*}
     \int_{\partial B(r)}kd\pi_r  
    &\leq& C\frac{\displaystyle r^{1-2m}\left(\displaystyle\int_{\frac{1}{3}}^rG(t)^{1-2m}dt\right)^{(1+\delta)^2}}{G(r)^{(1-2m)(1+\delta)}}\Lambda(r)^{(1+\delta)^2} \\
     &\leq& CK(r,\delta)\Gamma(r)^{(1+\delta)^2}
               \end{eqnarray*}
  holds for  all $r>1$ outside $E_\delta,$ where $C=1/c\omega_{2m-1}>0$ is  clearly a constant independent of  $\delta, r.$ 
  This completes the proof. 
\end{proof}

By  estimating $\log^+K(r,\delta),$ we further obtain: 

\begin{cor}\label{cal1} Let $k\geq0$ be a locally integrable function on $M.$ Assume that $k$ is locally bounded at $o.$ Then 
 for any   $\delta>0,$  there exists   a subset $E_{\delta}\subset(1,\infty)$ of finite Lebesgue measure such that
 \begin{eqnarray*}
&& \log^+\int_{\partial B(r)}kd\pi_r  \\
&\leq& 
      (1+\delta)^2\log^+\int_{B(r)}g_r(o,x)kdv+O\left(\sqrt{-\kappa(r)}r+\delta\log r\right)
      \end{eqnarray*} 
holds for all  $r>1$ outside $E_{\delta},$  where $\kappa(r)$ is defined by $(\ref{ricci}).$     
 \end{cor}
\begin{proof}  
 By  Lemma \ref{est}, we have  $G(t)\geq t$ for $t>0.$ Thus, one has
   \begin{eqnarray*}
\int_{\frac{1}{3}}^rG(t)^{1-2m}dt &\leq& \int_{1}^r t^{1-2m}dt+O(1) \\
&\leq& \log r+O(1)
\end{eqnarray*}
for $r>1.$
It yields   that 
\begin{equation}\label{x1}
\log^+\left(\int_{\frac{1}{3}}^rG(t)^{1-2m}dt\right)^{(1+\delta)^2}\leq (1+\delta)^2\log^+\log r+O(1).
\end{equation}
If $\kappa(r)\equiv0,$ then $G(r)\leq r$ due to Lemma \ref{est}. Thus, we obtain  
$$r^{1-2m}G(r)^{(2m-1)(1+\delta)}\leq r^{(2m-1)\delta},$$
which leads to  
\begin{equation}\label{x2}
\log^+\left(r^{1-2m}G(r)^{(2m-1)(1+\delta)}\right)\leq (2m-1)\delta \log r
\end{equation}
for $r>1.$ Combining (\ref{x1}) with (\ref{x2}), we conclude that 
$$\log^+K(r,\delta)\leq O(\delta\log r).$$
Using Theorem \ref{cal}, we have  the conclusion holds. 
If $\kappa(r)\not\equiv0,$ then  Lemma \ref{est}  yields that 
\begin{eqnarray}\label{x3}
 \log^+\left(r^{1-2m}G(r)^{(2m-1)(1+\delta)}\right) 
&\leq&\log^+G(r)^{(2m-1)(1+\delta)}   \\
&\leq& O\left(\sqrt{-\kappa(r)}r+1\right). \nonumber
\end{eqnarray}
Combining (\ref{x1}) with (\ref{x3}), we conclude that 
$$\log^+K(r,\delta)\leq O\left(\sqrt{-\kappa(r)}r+1\right).$$
Hence,  the conclusion  holds also due to   Theorem \ref{cal}.  
\end{proof}

In the following, we shall establish a logarithmic derivative lemma.     Let $\psi$ be a meromorphic function on $M.$ In terms of   holomorphic local coordinates $z_1,\cdots,z_m,$
the norm of the gradient $\nabla\psi$  is defined  by 
$$\|\nabla\psi\|^2=2\sum_{i,j=1}^m g^{i\overline j}\frac{\partial\psi}{\partial z_i}\overline{\frac{\partial \psi}{\partial  z_j}},$$
 where $(g^{i \bar j})$ is the inverse of $(g_{i \bar j}).$
Define   
$$T(r,\psi):=m(r,\psi)+N(r,\psi),$$ where 
\begin{eqnarray*}
m(r,\psi)&=&\int_{\partial B(r)}\log^+|\psi|d\pi_r, \\
N(r,\psi)&=& \frac{\pi^m}{(m-1)!}\int_{\psi^*\infty\cap B(r)}g_r(o,x)\alpha^{m-1}.
 \end{eqnarray*}
\ \ \ \    On $\mathbb P^1(\mathbb C),$ we take a singular metric
$$\Psi=\frac{1}{|\zeta|^2(1+\log^2|\zeta|)}\frac{\sqrt{-1}}{4\pi^2}d\zeta\wedge d\bar \zeta$$  
so that 
   $$ \int_{\mathbb P^1(\mathbb C)}\Psi=1.$$

\begin{lemma}\label{oo12} Let $\psi\not\equiv0$ be a  meromorphic function on  $M.$  Then 
$$\frac{1}{4\pi}\int_{B(r)}g_r(o,x)\frac{\|\nabla\psi\|^2}{|\psi|^2(1+\log^2|\psi|)}dv\leq T(r,\psi)+O(1).$$
\end{lemma}
\begin{proof}  It is not hard to deduce that   
$$\frac{\|\nabla\psi\|^2}{|\psi|^2(1+\log^2|\psi|)}=4m\pi\frac{\psi^*\Psi\wedge\alpha^{m-1}}{\alpha^m}.$$
 Using     
 Fubini's theorem, we conclude that  
\begin{eqnarray*}
&& \frac{1}{4\pi}\int_{B(r)}g_r(o,x)\frac{\|\nabla\psi\|^2}{|\psi|^2(1+\log^2|\psi|)}dv \\ 
&=&m\int_{B(r)}g_r(o,x)\frac{\psi^*\Psi\wedge\alpha^{m-1}}{\alpha^m}dv  \\
&=&\frac{\pi^m}{(m-1)!}\int_{\mathbb P^1(\mathbb C)}\Psi(\zeta)\int_{\psi^*\zeta\cap B(r)}g_r(o,x)\alpha^{m-1} \\
&=&\int_{\mathbb P^1(\mathbb C)}N\Big(r, \frac{1}{\psi-\zeta}\Big)\Psi(\zeta)  \\
&\leq&\int_{\mathbb P^1(\mathbb C)}\big{(}T(r,\psi)+O(1)\big{)}\Psi \\
&=& T(r,\psi)+O(1). 
\end{eqnarray*}
\end{proof}

\begin{lemma}\label{999a}  Let
$\psi\not\equiv0$ be a  meromorphic function on  $M.$  Then for any   $\delta>0,$ there exists a subset 
 $E_\delta\subset(1,\infty)$ of finite Lebesgue measure such that
\begin{eqnarray*}
  &&  \int_{\partial B(r)}\log^+\frac{\|\nabla\psi\|^2}{|\psi|^2(1+\log^2|\psi|)}d\pi_r \\
   &\leq& (1+\delta)^2 \log^+T(r,\psi)+O\left(\sqrt{-\kappa(r)}r+\delta\log r\right)
   \end{eqnarray*}
 holds for  all $r>1$ outside  $E_\delta,$ where $\kappa(r)$ is defined by $(\ref{ricci}).$      
\end{lemma}
\begin{proof} The concavity of $``\log"$ implies that    
\begin{eqnarray*}
&&  \int_{\partial B(r)}\log^+\frac{\|\nabla\psi\|^2}{|\psi|^2(1+\log^2|\psi|)}d\pi_r  \\
    &\leq&  \log^+\int_{\partial B(r)}\frac{\|\nabla\psi\|^2}{|\psi|^2(1+\log^2|\psi|)}d\pi_r+O(1).
\end{eqnarray*}
 Apply  Theorem \ref{cal} and Lemma \ref{oo12} to the first term on the right hand side of the above inequality, 
then  for any $\delta>0,$ there exists a subset 
 $E_\delta\subset(1,\infty)$ of finite Lebesgue measure such that  this term is bounded from above by 
    $$(1+\delta)^2 \log^+T(r,\psi)
    +O\left(\sqrt{-\kappa(r)}r+\delta\log r\right)$$
  for all $r>1$ outside  $E_\delta.$ This completes the proof. 
\end{proof}
Define
$$m\left(r,\frac{\|\nabla\psi\|}{|\psi|}\right)=\int_{\partial B(r)}\log^+\frac{\|\nabla\psi\|}{|\psi|}d\pi_r.$$

  We establish a logarithmic derivative lemma: 

\begin{theorem}[Logarithmic Derivative  Lemma]\label{log} Let
$\psi\not\equiv0$ be a  meromorphic function on  $M.$   Then for any   $\delta>0,$ there exists a  subset  $E_\delta\subset(1,\infty)$ of  finite Lebesgue measure such that 
\begin{eqnarray*}
   m\Big(r,\frac{\|\nabla\psi\|}{|\psi|}\Big)&\leq& \frac{2+(1+\delta)^2}{2}\log^+ T(r,\psi) 
    +O\left(\sqrt{-\kappa(r)}r+\delta\log r\right)
\end{eqnarray*}
 holds for   $r>1$ outside  $E_\delta,$ where $\kappa(r)$ is defined by $(\ref{ricci}).$     
\end{theorem}
\begin{proof}  We have 
\begin{eqnarray*}
    m\left(r,\frac{\|\nabla\psi\|}{|\psi|}\right)  
   &\leq& \frac{1}{2}\int_{\partial B(r)}\log^+\frac{\|\nabla\psi\|^2}{|\psi|^2(1+\log^2|\psi|)}d\pi_r   \\ 
 &&   +\frac{1}{2}\int_{\partial B(r)}\log\left(1+\log^2|\psi|\right)d\pi_r \\
        &\leq&  \frac{1}{2}\int_{\partial B(r)}\log^+\frac{\|\nabla\psi\|^2}{|\psi|^2(1+\log^2|\psi|)}d\pi_r  \\ 
   && +\log\int_{\partial B(r)}\Big{(}\log^+|\psi|+\log^+\frac{1}{|\psi|}\Big{)}d\pi_r +O(1)  \\
   &\leq& \frac{1}{2}\int_{\partial B(r)}\log^+\frac{\|\nabla\psi\|^2}{|\psi|^2(1+\log^2|\psi|)}d\pi_r+\log^+T(r,\psi).
\end{eqnarray*}
Apply  Lemma \ref{999a} again,  we have the theorem proved.
   \end{proof}

\section{Proofs of  Theorem \ref{main1} and Corollary \ref{main2}}

Let $E=\sum_j\mu_jE_j$ be a divisor, where  $E_j^,s$ are prime divisors.  The reduced form of $E$ is  defined by    
$${\rm Red}(E):=\sum_jE_j.$$

\emph{Proof of Theorem $\ref{main1}$}

 Write $D=D_1+\cdots+D_q$ as  the irreducible decomposition of $D.$
Equipping every holomorphic line bundle $\mathscr O(D_j)$ with a Hermitian
metric $h_j$ such that it  induces  the  Hermitian metric $h=h_1\otimes\cdots\otimes h_q$ on $L.$ 
Pick $s_j\in H^0(X, \mathscr O(D_j)$
such that  $(s_j)=D_j$ and $\|s_j\|<1.$
On $X,$ define a singular volume form
$$
  \Phi=\frac{\wedge^nc_1(L,h)}{\prod_{j=1}^q\|s_j\|^2}.$$
Set
$$f^*\Phi\wedge\alpha^{m-n}=\xi\alpha^m.$$
It is not hard to deduce     
$$dd^c[\log\xi]\geq f^*c_1(L, h_L)-f^*{\rm{Ric}}(\Omega)+\mathscr{R}-\left[{\rm{Red}}(f^*D)\right]$$
in the sense of currents.  
Hence, it  yields   that
\begin{eqnarray}\label{5q}
&& \frac{1}{4}\int_{B(r)}g_r(o,x)\Delta\log\xi dv \\
&\geq& T_{f}(r,L)+T_{f}(r,K_X)+T(r,\mathscr{R})-\overline{N}_{f}(r,D).  \nonumber
\end{eqnarray}
 \ \ \ \   On the other hand,   since $D$  only  has  simple normal crossings, then there exist  
  a finite  open covering $\{U_\lambda\}$ of $X,$ and finitely many   rational functions
$w_{\lambda1},\cdots,w_{\lambda n}$ on $X$ for each $\lambda,$  such that all $w_{\lambda j}^,s$ are holomorphic on $U_\lambda$  and     
\begin{eqnarray*}
  dw_{\lambda1}\wedge\cdots\wedge dw_{\lambda n}(x)\neq0, & & \ \ \    ^\forall x\in U_{\lambda}; \\
  D\cap U_{\lambda}=\big{\{}w_{\lambda1}\cdots w_{\lambda h_\lambda}=0\big{\}}, && \ \ \   ^\exists h_{\lambda}\leq n.
\end{eqnarray*}
In addition, we  can require that  $\mathscr O(D_j)|_{U_\lambda}\cong U_\lambda\times \mathbb C$ for 
all $\lambda,j.$ On  $U_\lambda,$   write 
$$\Phi=\frac{e_\lambda}{|w_{\lambda1}|^2\cdots|w_{\lambda h_{\lambda}}|^2}
\bigwedge_{k=1}^n\frac{\sqrt{-1}}{\pi}dw_{\lambda k}\wedge d\bar w_{\lambda k},$$
where  $e_\lambda$ is a  positive smooth function on $U_\lambda.$  
Let  $\{\phi_\lambda\}$ be a partition of the unity subordinate to $\{U_\lambda\}.$ Set 
$$\Phi_\lambda=\frac{\phi_\lambda e_\lambda}{|w_{\lambda1}|^2\cdots|w_{\lambda h_{\lambda}}|^2}
\bigwedge_{k=1}^n\frac{\sqrt{-1}}{\pi}dw_{\lambda k}\wedge d\bar w_{\lambda k}.$$
 Again, put $f_{\lambda k}=w_{\lambda k}\circ f.$  On  $f^{-1}(U_\lambda),$ we have  
\begin{eqnarray*}
 f^*\Phi_\lambda&=&
   \frac{\phi_{\lambda}\circ f\cdot e_\lambda\circ f}{|f_{\lambda1}|^2\cdots|f_{\lambda h_{\lambda}}|^2}
   \bigwedge_{k=1}^n\frac{\sqrt{-1}}{\pi}df_{\lambda k}\wedge d\bar f_{\lambda k} \\
   &=& \phi_{\lambda}\circ f\cdot e_\lambda\circ f\sum_{1\leq i_1\not=\cdots\not= i_n\leq m}
   \frac{\Big|\frac{\partial f_{\lambda1}}{\partial z_{i_1}}\Big|^2}{|f_{\lambda 1}|^2}\cdots 
   \frac{\Big|\frac{\partial f_{\lambda h_\lambda}}{\partial z_{i_{h_\lambda}}}\Big|^2}{|f_{\lambda h_\lambda}|^2}
   \left|\frac{\partial f_{\lambda (h_\lambda+1)}}{\partial z_{i_{h_\lambda+1}}}\right|^2 \\
 && \ \ \ \ \ \ \ \ \ \  \  \   \cdots\left|\frac{\partial f_{\lambda n}}{\partial z_{i_{n}}}\right|^2 
    \Big(\frac{\sqrt{-1}}{\pi}\Big)^ndz_{i_1}\wedge d\bar z_{i_1}\wedge\cdots\wedge dz_{i_n}\wedge d\bar z_{i_n}.
\end{eqnarray*}
 Fix an arbitrary   $x_0\in M.$  We  take  holomorphic local  coordinates $z_1,\cdots,z_m$ near $x_0$ and   holomorphic local  coordinates
  $\zeta_1,\cdots,\zeta_n$ 
near $f(x_0)$ such that
$$ \alpha|_{x_0}=\frac{\sqrt{-1}}{\pi}\sum_{j=1}^m dz_j\wedge d\bar{z}_j$$
and 
 $$c_1(L, h)\big|_{f(x_0)}=\frac{\sqrt{-1}}{\pi}\sum_{j=1}^n d\zeta_j\wedge d\bar{\zeta}_j.$$
 Set   
$$f^*\Phi_\lambda\wedge\alpha^{m-n}=\xi_\lambda\alpha^m.$$
Then, we have  $\xi=\sum_\lambda \xi_\lambda$ and   
 \begin{eqnarray*}
&& \xi_\lambda|_{x_0} \\ 
&=& \phi_{\lambda}\circ f\cdot e_\lambda\circ f\sum_{1\leq i_1\not=\cdots\not= i_n\leq m}
   \frac{\Big|\frac{\partial f_{\lambda1}}{\partial z_{i_1}}\Big|^2}{|f_{\lambda 1}|^2}\cdots 
   \frac{\Big|\frac{\partial f_{\lambda h_\lambda}}{\partial z_{i_{h_\lambda}}}\Big|^2}{|f_{\lambda h_\lambda}|^2}
   \left|\frac{\partial f_{\lambda (h_\lambda+1)}}{\partial z_{i_{h_\lambda+1}}}\right|^2\cdots\left|\frac{\partial f_{\lambda n}}{\partial z_{i_{n}}}\right|^2 \\
   &\leq&  \phi_{\lambda}\circ f\cdot e_\lambda\circ f\sum_{1\leq i_1\not=\cdots\not= i_n\leq m}
    \frac{\big\|\nabla f_{\lambda1}\big\|^2}{|f_{\lambda 1}|^2}\cdots 
   \frac{\big\|\nabla f_{\lambda h_\lambda}\big\|^2}{|f_{\lambda h_\lambda}|^2} \\
   &&\ \ \ \ \ \ \ \ \ \ \ \ \ \ \ \ \ \ \ \ \ \ \ \ \ \  \  \ \ \   \cdot \big\|\nabla f_{\lambda(h_\lambda+1)}\big\|^2\cdots\big\|\nabla f_{\lambda n}\big\|^2.
\end{eqnarray*} 
Define a non-negative function $\varrho$ on $M$ by  
\begin{equation}\label{wer}
  f^*c_1(L, h)\wedge\alpha^{m-1}=\varrho\alpha^m.
\end{equation}
Again, put  $f_j=\zeta_j\circ f$ with  $j=1,\cdots,n.$  Then, we have      
\begin{equation*}
    f^*c_1(L, h)\wedge\alpha^{m-1}\big|_{x_0}=\frac{(m-1)!}{2}\sum_{j=1}^m\big\|\nabla f_j\big\|^2\alpha^m, 
\end{equation*}
which yields that   
$$\varrho|_{x_0}=(m-1)!\sum_{i=1}^n\sum_{j=1}^m\Big|\frac{\partial f_i}{\partial z_j}\Big|^2
=\frac{(m-1)!}{2}\sum_{j=1}^n\big\|\nabla f_j\big\|^2.$$
 Put together   the above, we are led to 
$$\xi_\lambda\leq 
\frac{ \phi_{\lambda}\circ f\cdot e_\lambda\circ f\cdot(2\varrho)^{n-h_\lambda}}{(m-1)!^{n-h_\lambda}}\sum_{1\leq i_1\not=\cdots\not= i_n\leq m}
    \frac{\big\|\nabla f_{\lambda1}\big\|^2}{|f_{\lambda 1}|^2}\cdots 
   \frac{\big\|\nabla f_{\lambda h_\lambda}\big\|^2}{|f_{\lambda h_\lambda}|^2}
$$
on $f^{-1}(U_\lambda).$
Since $\phi_\lambda\circ f\cdot e_\lambda\circ f$ is bounded on $M$ and   
$$\log^+\xi\leq \sum_\lambda\log^+\xi_\lambda+O(1),$$
 we obtain  
\begin{equation}\label{bbd}
   \log^+\xi\leq O\left(\log^+\varrho+\sum_{k, \lambda}\log^+\frac{\|\nabla f_{\lambda k}\|}{|f_{\lambda k}|}+1\right). 
 \end{equation}  
  Jensen-Dynkin formula yields  that 
\begin{equation*}
 \frac{1}{2}\int_{B(r)}g_r(o,x)\Delta\log\xi dv
=\int_{\partial B(r)}\log\xi d\pi_r+O(1).
\end{equation*}
Combining this with (\ref{bbd}) and Theorem \ref{log} to get  
\begin{eqnarray*}
&& \frac{1}{4}\int_{B(r)}g_r(o,x)\Delta\log\xi dv\\
   &\leq& O\left(\sum_{k,\lambda}m\left(r,\frac{\|\nabla f_{\lambda k}\|}{|f_{\lambda k}|}\right)+\log^+\int_{\partial B(r)}\varrho d\pi_r+1\right) \\
   &\leq& O\left(\sum_{k,\lambda}\log^+ T(r,f_{\lambda k})+\log^+\int_{\partial B(r)}\varrho d\pi_r+
    \sqrt{-\kappa(r)}r+\delta\log r\right) \\
      &\leq& O\left(\log^+ T_f(r,L)+\log^+\int_{\partial B(r)}\varrho d\pi_r+
     \sqrt{-\kappa(r)}r+\delta\log r\right). 
   \end{eqnarray*}
Using Theorem \ref{cal} and (\ref{wer}),  for any $\delta>0,$ there exists a subset $E_\delta\subset(0,\infty)$ of finite Lebesgue measure such that 
   \begin{eqnarray*}
 \log^+\int_{\partial B(r)}\varrho d\pi_r 
   &\leq&  (1+\delta)^2\log^+ T_f(r,L)+
     O\left(\sqrt{-\kappa(r)}r+\delta\log r\right)
      \end{eqnarray*}
   holds for  all $r>1$ outside $E_\delta.$  
It is therefore   
\begin{eqnarray*}\label{6q}
     \frac{1}{4}\int_{B(r)}g_r(o,x)\Delta\log\xi dv 
 &\leq&  O\left(\log^+ T_f(r,L)+
      \sqrt{-\kappa(r)}r+\delta\log r\right) 
\end{eqnarray*}
 for  $r>1$ outside $E_\delta.$  Combining  this with   (\ref{5q}),  we prove the theorem. Q.E.D.

The following    volume  comparison theorem is due  to Bishop-Gromov (see, e.g., \cite{B, S-Y}). 

\begin{lemma}\label{comp}  Let $N$ be a complete Riemannian manifold.  Let $B(x, r)$ denote  the geodesic ball centered at $x$ with radius $r$ in $N.$  If ${\rm{Ric}}(N)\geq(n-1)\kappa$ for some constant $\kappa,$  then 
$${\rm{Area}}\left(\partial B(x, r)\right)\leq {\rm{Area}}\left(\partial B(\kappa, r)\right)$$
and thus 
$${\rm{Vol}}\left(B(x, r)\right)\leq {\rm{Vol}}\left(B(\kappa, r)\right),$$
where $B(\kappa, r)$ denotes  a  geodesic ball with radius $r$ in the space form with  sectional curvature $\kappa,$ of  the same dimension as one of $N.$  
\end{lemma}

\begin{remark}\label{remk} Let $B(\kappa, r)$ be denoted  as  in Lemma \ref{comp}.  The  following fact is well known: 
$${\rm{Area}}\left(\partial B(\kappa, r)\right)=\left\{
                \begin{array}{ll}
               \omega_{n-1}\chi(\sqrt{-\kappa}, r)^{n-1}, \  \   & \kappa\leq0; \\
                  \omega_{n-1}\Big(\frac{\sin\sqrt{\kappa}r}{\sqrt{\kappa}}\Big)^{n-1},  \ \     & \kappa>0, \\
                \end{array}
              \right. $$
where $\omega_{n-1}$ denotes  the standard Euclidean area of the unit sphere in $\mathbb R^n,$ and $\chi(s,t)$ is defined by $(\ref{chi}).$
\end{remark}

Let $N$ be a complete non-compact Riemannian manifold with a pole $o,$ of dimension $n.$  Assume that the sectional curvature $K$ and the  Ricci curvature ${\rm{Ric}}$ of $N$ satisfy  that 
\begin{equation*}
 K\leq-\sigma, \ \ \ \ {\rm{Ric}}\geq-(n-1)\tau
 \end{equation*}
 for  some  constants $\sigma,\tau\geq0.$ It is evident that   $\sigma\leq\tau.$
Let $G(o,x)$ denote the minimal positive Green function of $\Delta/2$ for $N$ with a pole at $o,$ where   $\Delta$ is the  Laplace-Beltrami operator on $N.$ 

Let $G(t), H(t)$  solve    the following  Jacobi equations  on $[0,\infty)$:
 $$G''-\sigma G=0; \ \ \ \    G(0)=0, \ \ \ \  G'(0)=1;$$ 
 $$H''-\tau H=0; \ \ \ \    H(0)=0, \ \ \ \  H'(0)=1,$$
respectively.
 Employing  Laplacian comparison theorem \cite{Greene}, we deduce that   (see \cite{Ka0, Sa0} also)   
 \begin{equation}\label{greenn}
\frac{2}{\omega_{2m-1}}\int_{\rho(x)}^\infty H(t)^{1-2m}dt\leq G(o,x)\leq \frac{2}{\omega_{2m-1}}\int_{\rho(x)}^\infty G(t)^{1-2m}dt,
\end{equation} 
where  $\rho(x)$ is the Riemannian distance function of  $x$ form $o,$ and  $\omega_{n-1}$ is  the standard  Euclidean  area  of   the unit sphere in $\mathbb R^n.$

\emph{Proof of  Corollary $\ref{main2}$}

By  curvature assumption, we have  $\kappa(r)\equiv\kappa.$  Thus,  according to Theorem \ref{main1},  it suffices to show that 
\begin{equation}\label{final}
|T(r, \mathscr R)|\leq O\left(\sqrt{-\kappa}r\right).
\end{equation}
Let $s$ denote  the scalar curvature of $M.$ Then 
$$s=-\frac{1}{2}\Delta\log\det(g_{i\bar j}),$$
which is a non-positive constant.  
Let $G(o,x)$ be the minimal positive Green function of $\Delta/2$ for $M$ with a pole at $o.$ 
Let us consider     the following  Jacobi equation on $[0,\infty)$:
 $$G''+\kappa G=0; \ \ \ \    G(0)=0, \ \ \ \  G'(0)=1.$$
 It is uniquely  solved by $G(t)=\chi(\sqrt{-\kappa}, t).$ Since  $M$ has  non-positive constant sectional curvature $\kappa,$  
 it yields  from (\ref{greenn}) that 
     \begin{eqnarray*}
G(o,x)&=&\frac{2}{\omega_{2m-1}}\int_{\rho(x)}^\infty G(t)^{1-2m}dt \\
&=&\frac{2}{\omega_{2m-1}}\int_{\rho(x)}^\infty \chi(\sqrt{-\kappa}, t)^{1-2m}dt.
     \end{eqnarray*}
   By  
   $$G(o,x)\big|_{\partial B(r)}=\frac{2}{\omega_{2m-1}}\int_{r}^\infty \chi(\sqrt{-\kappa}, t)^{1-2m}dt,$$
it gives    
     \begin{eqnarray*}
     g_r(o,x)&=&G(o,x)-G(o,x)\big|_{\partial B(r)} \\
     &=&\frac{2}{\omega_{2m-1}}\int_{\rho(x)}^r \chi(\sqrt{-\kappa}, t)^{1-2m}dt.
          \end{eqnarray*}
Combined with  Lemma \ref{comp} and Remark \ref{remk}, we obtain     
     \begin{eqnarray*}
 |T(r,\mathscr R)| &=&  \frac{1}{4}\int_{B(r)}g_r(o,x)\Delta\log\det(g_{i\bar j})dv \\
 &\leq& -2s\int_{B(r)}g_r(o,x)dv \\
 &=& -2s\int_0^rdt\int_{\partial B(t)}g_r(o,x)d\sigma_t \\
 &=& -\frac{4s}{\omega_{2m-1}}\int_0^rdt\int_{t}^r \chi(\sqrt{-\kappa}, s)^{1-2m}ds\int_{\partial B(t)}d\sigma_t \\
 &\leq& -4s\int_0^r\chi(\sqrt{-\kappa}, t)^{2m-1}dt\int_t^r \chi(\sqrt{-\kappa}, s)^{1-2m}ds. 
     \end{eqnarray*}
If $\kappa=0,$ then (\ref{final}) holds clearly due to $\mathscr R=0.$  In the following,  one assumes   that $\kappa<0.$ 
 Since
     \begin{eqnarray*}
\int_t^r \chi(\sqrt{-\kappa}, s)^{1-2m}ds &\leq& O\left(\int_t^r e^{(1-2m)\sqrt{-\kappa}s}ds\right) \\
&\leq& O\left(e^{(1-2m)\sqrt{-\kappa}t}\right), 
     \end{eqnarray*}
we have 
     \begin{eqnarray*}
&& \int_0^r\chi(\sqrt{-\kappa}, t)^{2m-1}dt\int_t^r \chi(\sqrt{-\kappa}, s)^{1-2m}ds \\
&\leq& \frac{1}{(2\sqrt{-\kappa})^{2m-1}} \int_0^r e^{(2m-1)\sqrt{-\kappa}t}dt\int_t^r \chi(\sqrt{-\kappa}, s)^{1-2m}ds \\
&\leq& O\left(\int_0^rdt\right) \\
&=&O(r).
     \end{eqnarray*}
     Thus, we conclude that 
     $$|T(r, \mathscr R)|\leq O(r).$$
To sum up, we have shown that $(\ref{final})$ holds.  
This completes the proof. Q.E.D.
 
% \vskip\baselineskip

%\noindent\textbf{Acknowledgements.}  
%  This research work was partially  supported  by the NSF of Shandong Province of China
% (ZR202211290346).

\vskip\baselineskip

\vskip\baselineskip

\end{document}